\documentclass[10pt]{amsart}
\usepackage{enumerate,amsmath,amssymb,latexsym,
amsfonts, amsthm, amscd, mathabx, MnSymbol}

\theoremstyle{plain}

\newtheorem{theorem}{Theorem}

\numberwithin{equation}{section}

\newcommand{\C}{\mathbb{C}}
\newcommand{\E}{\mathbb{E}}
\newcommand{\R}{\mathbb{R}}

\newcommand{\sm}{\rm sm}
\newcommand{\cm}{\rm cm}

\addtocounter{section}{-1}

\begin{document}

\title {The Dixonian elliptic functions}

\date{}

\author[P.L. Robinson]{P.L. Robinson}

\address{Department of Mathematics \\ University of Florida \\ Gainesville FL 32611  USA }

\email[]{paulr@ufl.edu}

 \keywords{}

\begin{abstract}

We offer a careful development of the Dixonian elliptic functions with parameter $\alpha = 0$ from the initial value problem of which they are solutions. 

\end{abstract}

\maketitle

\section*{Introduction} 

\medbreak 

A.C. Dixon [2] presented a detailed account of the elliptic functions that develop from the cubic curve $x^3 + y^3 - 3 \alpha x y = 1$ by the inversion of associated Abelian integrals. These particular elliptic functions seem thereafter to have been largely neglected, save for applications to geodesy, until it became clear that they have interesting connexions to combinatorics and probability [1] and provide arc-length parametrizations for certain sextic curves [4]. These more recent developments involve the specific case of the Fermat cubic, for which $\alpha = 0$: the elliptic functions $\cm$ and $\sm$ satisfy 
$$\cm^3 + \sm^3 = 1$$
and (when suitably interpreted)
$$z = \int_0^{\sm (z)} \frac{{\rm d} \sigma}{(1 - \sigma^3)^{2/3}} = \int_{\cm (z)}^1 \frac{{\rm d} \sigma}{(1 - \sigma^3)^{2/3}} .$$

\bigbreak 

The authors of [1] remark (on page 6) that `It is fascinating to be able to develop a fair amount of the theory from the differential
equation (7)' where by `(7)' is meant the following initial value problem: 
$$\cm' = - \sm^2, \; \sm' = \cm^2; \;\; \cm(0) = 1, \; \sm(0) = 0.$$
We here take this remark quite seriously. According to the famous Picard existence-uniqueness theorem, there exists a unique solution to this problem on a disc about $0$ having sufficiently small radius, which we specify. Here we show that these local holomorphic functions extend to global meromorphic functions, which are elliptic. Much of their structure may be read from our account, including their periods, their zeros and poles, and their addition formulae. All of this we achieve without the aid of any transcendental functions, using little more than the Picard existence-uniqueness theorem for first-order systems and the Identity Theorem (or `principle of analytic continuation'); at least until it becomes convenient to use the Schwarz Symmetry (or Reflexion) Principle to complete the construction. Our account is intended in part as an introduction to the Dixonian functions, perhaps with one or two new proofs. 

\medbreak 

\section*{The Dixonian elliptic functions}

\medbreak 

We start from the following initial value problem (henceforth referred to as {\bf IVP}): 
$$c\,' = - s^2, \; s\,' = c^2; \;\; c(0) = 1, \; s(0) = 0.$$
Here, solutions $c$ and $s$ to this system are sought as holomorphic functions on a connected open set containing the origin, this domain being pushed as far as is possible into the complex plane. In order to begin the process, we apply the classical existence-uniqueness theorem due to Picard. We may of course apply this Picard theorem in its `system' form. Instead, we choose to establish the local existence of a solution to {\bf IVP} by first solving the single initial value problem 
$$s\;' = (1 - s^3)^{2/3} ; \; \; s(0) = 0$$
and then defining $c = (1 - s^3)^{1/3}$; here, we assign to powers their principal values, at least initially. For a discussion of the Picard theorem and its relatives, see Chapter 2 of [3]: the Picard theorem appears in its single form as Theorem 2.3.1 and in its system form as Theorem 2.3.2. 

\medbreak 

\begin{theorem} \label{init}
There exists a unique holomorphic function $s : B_{2^{-2/3}} (0) \to \C$ such that $s(0) = 0$ and $s\,' = (1 - s^3)^{2/3}$.
\end{theorem} 

\begin{proof} 
Adopting the notation of [3] Theorem 2.3.1, we apply the Picard theorem to the IVP 
$$w\,' = F(z, w)  ; \; \; w(z_0) = w_0$$
with $F(z, w) = (1 - w^3)^{2/3}$, $z_0 = 0$ and $w_0 = 0$. Let $0 < b < 1$: if $z$ is arbitrary and $|w| < b$ then $F(z, w)$, unambiguously given by the principal-valued power, satisfies $|F(z, w)| < M = (1 + b^3)^{2/3}$ and a Lipschitz condition, so the IVP above has a unique solution defined in the disc $B_r (0)$ with $r = b/M = b/(1 + b^3)^{2/3}$; and this solution is valued in the open disc $B_b (0)$. We may optimize $r$ by taking $b \uparrow 1$: it follows that the IVP has a unique solution in $B_r (0)$ with $r = 2^{-2/3}$ as claimed; and this solution is valued in the open unit disc.  

\end{proof} 

\medbreak 

As the holomorphic function $s$ takes its values in the open unit disc, we may define a holomorphic function $c: B_{2^{-2/3}} (0) \to \C$ by $c = (1 - s^3)^{1/3}$ with principal-valued power.  

\medbreak 

Henceforward, let us agree to write $r = 2^{-2/3} = 0.62996 ...$ for convenience. 

\medbreak 

\begin{theorem} \label{pair}
The pair $(c, s)$ is the unique holomorphic solution to {\bf IVP} in the open disc $ B_r (0).$
\end{theorem} 

\begin{proof} 
 Simply note, further to $s(0) = 0$ and $s\,' = c^2$, that $c(0) = 1$ and that 
$$c\,' = (1/3) (-3 s^2 s\,') (1 - s^3)^{-2/3} = -s^2.$$
Of course, the solution is unique. 
\end{proof}

We have thus established the unique existence of a solution $(c, s)$ to the system {\bf IVP} on the open disc $B_r (0)$ of radius $r$ about $0$. This solution satisfies $c^3 + s^3 = 1$, by construction and indeed as a direct result of {\bf IVP}.  

\medbreak 

The functions $c$ and $s$ satisfy certain symmetries. First, each function has real output for real input: in fact, each is `real' in the following sense. 

\medbreak 

\begin{theorem} \label{real}
If $z \in B_r (0)$ then $\overline{c(z)} = c(\overline{z})$ and $\overline{s(z)} = s(\overline{z}).$
\end{theorem} 

\begin{proof} 
Define functions $C$ and $S$ in $B_r (0)$ by $C(z) = \overline{c(\overline{z})}$ and $S(z) = \overline{S(\overline{z})}$. It is straightforward to verify that the pair $(C, S)$ satisfies {\bf IVP} and so equals $(c, s)$ on account of Theorem \ref{pair}. 
\end{proof} 

\medbreak 

When their argument is reversed, the functions $c$ and $s$ behave as follows. 

\medbreak 

\begin{theorem} \label{neg}
If $z \in B_r (0)$ then $c(-z) = 1/c(z)$ and $s(-z) = -s(z) / c(z)$. 
\end{theorem} 

\begin{proof} 
Define functions $C$ and $S$ in $B_r (0)$ by $C(z) = 1/c(-z)$ and $S(z) = - s(-z) / c(-z)$. Straightforwardly,  $(C, S)$ satisfies {\bf IVP} and so coincides with $(c, s)$ by Theorem \ref{pair}. 
\end{proof} 

\medbreak 

Next, the functions $c$ and $s$ have certain three-fold symmetries. Write 
$$\gamma = \exp(2 \pi i /3) = (-1 + i \sqrt{3}) / 2$$
so that $\{ 1, \gamma, \overline{\gamma} = \gamma^{-1} = \gamma^2 \}$ is the set of complex cube-roots of unity. 

\medbreak 

\begin{theorem} \label{gamma}
If $z \in B_r (0)$ then $s(\gamma z) = \gamma \: s(z)$ and $c(\gamma z) = c(z)$. 
\end{theorem} 

\begin{proof} 
Define functions $C$ and $S$ in $B_r (0)$ by $C(z) = c(\gamma z)$ and $S(z) = \overline{\gamma} \:s(\gamma z)$. It is again straightforward to check that $(C, S)$ satisfies {\bf IVP} so that $(C, S) = (c, s)$ by Theorem \ref{pair}. 
\end{proof} 

\medbreak 

It is perhaps worth remarking that these three-fold symmetries of $c$ and $s$ serve as Dixonian counterparts of `parity'. 

\medbreak 

From Theorem \ref{gamma}, it follows that $c(z)$ is real not only when $z$ is real but also when $z \in \gamma \: \R$ and when $z \in \overline{\gamma} \: \R$; likewise, if $z \in \R$ then $s(z) \in \R$ while if $z \in \gamma^{\pm 1} \: \R$ then $s(z) \in \gamma^{\pm 1} \: \R$ respectively. Note also that Theorem \ref{real}, Theorem \ref{neg} and Theorem \ref{gamma} together yield the action of the cube-roots of $-1$ on $s$ and $c$: for example, 
$$s( - \gamma z) = - \gamma s(z)/c(z)\; \; {\rm and} \; \; c(- \gamma z) = 1 / c(z).$$

\medbreak

Now is perhaps as good a place as any to make contact with the integral formula that was stated in the introduction. 

\medbreak 

\begin{theorem} \label{integral}
If $z \in B_r (0)$ then 
$$\int_0^{\sm (z)} \frac{{\rm d} \sigma}{(1 - \sigma^3)^{2/3}} = z.$$
\end{theorem} 

\begin{proof} 
The integrand $\sigma \mapsto (1 - \sigma^3)^{-2/3}$ is certainly holomorphic in the open unit disc, to which $B_r (0)$ is mapped by $s$. Accordingly, we may consider the holomorphic function 
$$g: B_r (0) \to \C : z \mapsto \int_0^{\sm (z)} \frac{{\rm d} \sigma}{(1 - \sigma^3)^{2/3}}\, .$$
By Theorem \ref{init} we deduce that 
$$g\,' (z) = \frac{1}{(1 - s(z)^3)^{2/3}} \cdot s\,' (z) = 1$$
whence by $g(0) = 0$ we conclude that $g(z) = z$ whenever $z \in B_r (0)$. 
\end{proof} 

\medbreak 

Here, the integration may be effected along any contour in the open unit disc with the indicated endpoints.

\medbreak 

We may now determine precisely where in $B_r(0)$ the cube $s^3$ takes real values: Theorem \ref{real} and Theorem \ref{gamma} inform us that $s(z)^3$ is real when $z^3$ is real; an application of Theorem \ref{integral} yields the following converse.

\medbreak 

\begin{theorem} \label{reality}
If $z \in B_r(0)$ and $s(z)^3 \in \R$ then $s(z)$ is a positive real multiple of $z$. 
\end{theorem} 

\begin{proof} 
We may fix $z$ to be nonzero and in Theorem \ref{integral} integrate along the interval from $0$ to $s(z)$ by substituting $s(z) t$ for $\sigma$: there results 
$$z = \int_0^1 \frac{s(z) {\rm d} t}{(1 - s(z)^3 t^3)^{2/3}} = \Big[ \int_0^1 \frac{{\rm d} t}{(1 - s(z)^3 t^3)^{2/3}}\Big] \: s(z)$$
where the integral enclosed in brackets is a positive real since here $| s(z)| < 1$ and $s(z)^3 \in \R$ by assumption. 
\end{proof} 

\medbreak 

Thus, the cube $s^3$ takes real values precisely where its domain $B_r (0)$ meets the six rays through $0$ along which $\arg$ is a multiple of $\pi / 3$. 

\medbreak 

Our aim now is to extend the solution $(c, s)$ beyond the disc $B_r (0)$: as we shall see, the holomorphic functions $c$ and $s$ in $B_r (0)$ actually extend throughout the plane as meromorphic functions, each of which is doubly-periodic and therefore elliptic; further, the symmetry properties expressed in Theorem \ref{real}, Theorem \ref{neg} and Theorem \ref{gamma} propagate beyond $B_r (0)$ by the Identity Theorem. 

\medbreak 

The mechanism by which we propose to extend $c$ and $s$ beyond the initial disc may be traced to Weierstrass, who introduced it in his characterization of analytic functions that possess algebraic addition theorems; it also played a part in the masterly treatment [5] of Jacobian elliptic functions by Neville. In short, the mechanism rests on duplication. Each of the Dixonian functions $\cm$ and $\sm$ is known to satisfy a duplication formula involving only the functions $\cm$ and $\sm$ themselves: we take these duplication formulae as given, applying them to the functions $c$ and $s$ instead; we thereby extend $c$ and $s$ to the disc about $0$ having twice the radius. Repeated reduplication has to take account of poles, but the result is to extend $c$ and $s$ to functions that are meromorphic in the plane and still satisfy {\bf IVP}. Needless to say, only the shapes of the duplication formulae are taken as given; aside from this foreknowledge, everything rests ultimately upon and is developed rigorously from {\bf IVP}. 

\medbreak 

Continue to write $r = 2^{-2/3}$ for convenience. With $s$ and $c$ on $B_r (0)$ as given in Theorem \ref{init} and immediately thereafter, we define functions $S$ and $C$ on $B_{2 r} (0)$ by declaring that if $2 z \in B_{2 r} (0)$ then  
$$S(2 z) = \frac{s(z)(1 + c(z)^3)}{c(z)(1 + s(z)^3)}$$
and 
$$C(2 z) = \frac{c(z)^3 - s(z)^3}{c(z)(1 + s(z)^3)}$$
wherein the denominators are nonzero because $s$ takes its values in the open unit disc and $c^3 + s^3 = 1$. Each of the functions $S$ and $C$ is holomorphic in $B_{2 r} (0)$: in fact, elementary differentiation shows that $S\,' = C^2$ and $C\,' = -S^2$ while simple evaluation shows that $S(0) = 0$ and $C(0) = 1$. According to Theorem \ref{pair} (and so ultimately to the Picard theorem), the restrictions of $C$ and $S$ from $B_{2 r} (0)$ to $B_r (0)$ are $c$ and $s$ respectively. This proves that the unique solution $(c, s)$ to {\bf IVP} in the disc $B_r (0)$ of radius $r = 2^{-2/3}$ extends to a (unique) solution $(C, S)$ to {\bf IVP} in the disc $B_{2 r} (0)$ of radius $2 r = 2^{1/3}$. 

\medbreak 

As no confusion is likely to arise, we shall now drop the capitalization, referring to the extended functions simply as $c : B_{2 r}(0) \to \C$ and $s : B_{2 r}(0) \to \C$; this frees capitalization for repeated use in what follows. 

\medbreak 

We now pause to take stock of our findings. 

\medbreak 

\begin{theorem} \label{2 r}
The initial value problem {\bf IVP} has a unique holomorphic solution $(c, s)$ in $B_{2 r} (0)$.
\end{theorem} 
 
\begin{proof} 
Existence is established by the formulae for $C$ and $S$ displayed prior to the theorem; uniqueness is clear. 
\end{proof} 

\medbreak 

Notice that we have now established the following duplication formulae for $s$ and $c$: if $z \in B_r (0)$ then 
$$s(2 z) = \frac{s(z)(1 + c(z)^3)}{c(z)(1 + s(z)^3)}$$
and 
$$c(2 z) = \frac{c(z)^3 - s(z)^3}{c(z)(1 + s(z)^3)}.$$
As was promised above, these formulae have been established directly from the {\bf IVP} definition of $s$ and $c$; foreknowledge `only' saved us the work involved in discovering the formulae that must be established. 

\medbreak 

\begin{theorem} \label{zero}
The function $s$ has $0$ as its only zero in the disc $B_{2 r} (0)$. 
\end{theorem} 

\begin{proof} 
Deny: suppose that $z \in B_{2 r} (0)$ is nonzero but $s(z) = 0$. Recall that $1 + c^3$ is nowhere zero on the disc $B_r (0)$ and apply the duplication formula inductively in reverse: it follows that $s(2^{-n} z) = 0$ whenever $n$ is a natural number, so $0$ is a limit point of the zeros of $s$. This forces $s$ to be identically zero, which absurdity faults the supposition that $s$ has a nonzero zero. 
\end{proof} 

\medbreak 

As a corollary, $0$ is the only point of $B_{2 r} (0)$ at which the value of $c$ is a cube-root of unity. 

\medbreak 

\begin{theorem} \label{K}
The function $s$ is strictly increasing on the interval $(-2 r, 2 r)$. There exists a unique $K \in (0, 2)$ such that $s(-\tfrac{1}{2}K) = -1$. 
\end{theorem} 

\begin{proof} 
The fact that $s$ is strictly increasing is clear, for $s\,' = c^2>0$ throughout the interval. If $- 2 r < t < 0$ then $s(t) < 0$ so that $s\,' (t) = (1 - s(t)^3)^{2/3} > 1$ and therefore $s(t) < t$. As $- 2 r = - 2^{1/3} < -1$ we deduce that $s$ takes the value $-1$ at exactly one point of $(- 1 , 0)$; this point of $(-1, 0)$ we name $- \tfrac{1}{2} K$. 
\end{proof} 

\medbreak 

The symmetry of $s$ stated in Theorem \ref{gamma} implies that $s^3$ takes the value $-1$ not only at $-\tfrac{1}{2}K$ but also at $-\tfrac{1}{2} K \gamma$ and at $-\tfrac{1}{2} K \overline{\gamma}$. Notice that $c(-\tfrac{1}{2}K) = 2^{1/3}$ and $s(\tfrac{1}{2}K) = c(\tfrac{1}{2}K) = 2^{-1/3}$ by the identity $s^3 + c^3 = 1$ along with Theorem \ref{real} and Theorem \ref{neg} extended to $B_{2 r}(0)$. 

\medbreak 

We now consider, after this first duplication from $B_r (0)$ to $B_{2 r} (0)$, a second duplication from $B_{2 r} (0)$ to $B_{4 r} (0)$. Thus, for $z \in B_{2 r} (0)$ let us attempt the definitions 
$$S(2 z) = \frac{s(z)(1 + c(z)^3)}{c(z)(1 + s(z)^3)}$$
and 
$$C(2 z) = \frac{c(z)^3 - s(z)^3}{c(z)(1 + s(z)^3)}. $$
This attempt is successful in defining holomorphic functions $S$ and $C$ except at such points $z$ of $B_{2 r} (0)$ as satisfy either $c(z) = 0$ or $1 + s(z)^3 = 0$; equivalently, except at such points $z$ of $B_{2 r} (0)$ as satisfy $s(z)^3 = \pm 1$. We proceed to examine these two types of point separately. 

\medbreak 

Assume first that $c(z) = 0$ and write $z = 2 w$ with $w \in B_r (0)$. From the duplication formula for $c$ displayed after Theorem \ref{2 r}, it follows that $c(w)^3 - s(w)^3 = 0$ so that $c(w)^3 = s(w)^3 = 1/2$. From Theorem \ref{reality} it then follows that $w$ (or $\gamma w$ or $\overline{\gamma} w$) lies in the real interval $(0, r)$. However, if $0 < t < r$ then $0 < s(t) < 1$ so that $s\,'(t) = (1 - s(t)^3)^{2/3} < 1$ and therefore $0 < s(t) < t < r = 2^{-2/3} < 2^{-1/3}$. This rules out the existence of such a $w$ and hence of such a $z$. 

\medbreak 

Assume instead that $s(z)^3 = -1$ and again write $z = 2 w$ with $w \in B_r (0)$. For convenience, write $\sigma = s(w)$ and note that $\sigma$ lies in the open unit disc. From the duplication formula for $s$ displayed after Theorem \ref{2 r}, it follows that 
$$-1 = \Big(\frac{s(w)}{c(w)}\Big)^3\Big(\frac{1 + c(z)^3}{1 + s(z)^3}\Big)^3 = \frac{\sigma}{1 - \sigma} \frac{(2 - \sigma)^3}{(1 + \sigma)^3}$$
or 
$$0 = 1 + 10 \sigma - 12 \sigma^2 + 4 \sigma^3 - 2 \sigma^4.$$
This quartic is satisfied by only one value of $\sigma$ in the open unit disc: namely, 
$$\sigma = \Big(1 - \sqrt{3 (2\sqrt{3} - 3)}\, \Big) / 2 = -0.0899798 ... \, .$$
From Theorem \ref{reality} it follows that $w$ (or $\gamma w$ or $\overline{\gamma} w$) lies in the real interval $(-r, 0)$, on which $s$ is strictly increasing. This proves the existence of at most three such $w$ and hence at most three such $z$. 

\medbreak 

Of course, we are already in possession of three points $z \in B_{2 r} (0)$ such that $s(z)^3 = -1$: namely, $- \tfrac{1}{2} K$, $- \tfrac{1}{2} K \gamma$ and $- \tfrac{1}{2} K \overline{\gamma}$; so these are the precise points at which our attempt to define $S(2 z)$ and $C(2 z)$ is unsuccessful. 

\medbreak 

\begin{theorem} \label{4 r}
{\bf IVP} has a unique solution $(c, s)$ in $B_{4 r} (0)$, each of $c$ and $s$ being meromorphic, with simple poles at the points $-K$, $- K \gamma$ and $- K \overline{\gamma}$. 
\end{theorem} 

\begin{proof} 
Our duplication formulae have defined $C$ and $S$ as holomorphic functions on the disc $B_{4 r} (0)$ less the points $\{ -K, \, - K \gamma, \, - K \overline{\gamma} \}$. As for Theorem \ref{2 r}, direct calculation shows that the pair $(C, S)$ satisfies {\bf IVP} and therefore agrees on $B_{2 r} (0)$ with $(c, s)$ itself, again by the Picard uniqueness theorem and the Identity Theorem; on these grounds we again drop the capitalization, writing $(C, S)$ simply as $(c, s)$. Let $2 z$ be one of the three excluded points and refer to the duplication formulae that define $s(2 z) = S(2 z)$ and $c(2 z) = C(2 z)$: the numerator of each is nonzero, while the denominator has a simple zero; consequently, the excluded point is a simple pole of $s$ and of $c$. This attends to the question of existence; by now, uniqueness needs no attention. 
\end{proof} 

\medbreak 

Once again, our very construction has produced duplication formulae for $s$ and $c$. Of course, the various symmetries established for $s$ and $c$ in the initial disc $B_r (0)$ continue to hold in the disc $B_{4 r} (0)$. 

\medbreak 

The following relation between $s$ and $c$ calls to mind a similar property of the trigonometric functions. 

\medbreak 

\begin{theorem} \label{K -}
If $z \in B_{\tfrac{1}{2} K} (\tfrac{1}{2} K)$ then $s(K - z) = c(z)$ and $c(K - z) = s(z)$. 
\end{theorem} 

\begin{proof} 
Observe that the symmetry $z \mapsto K -z$ leaves the disc $B_{K/2} (K/2) \subseteq \E$ invariant. Define holomorphic functions $C$ and $S$ in the disc $B_{K/2} (K/2)$ by the rules $C(z) = s(K - z)$ and $S(z) = c(K - z)$. By differentiation, $C\,' = - S^2$ and $S\,' = C^2$, so $(C, S)$ and $(c, s)$ satisfy the same system of first-order differential equations in the disc $B_{K/2} (K/2)$. By evaluation, $C(K/2) = s(K/2) = c(K/2)$ and $S(K/2) = c(K/2) = s(K/2)$, so $(C, S)$ and $(c, s)$ satisfy the same initial conditions at the centre of the disc. By the Picard uniqueness theorem and the Identity Theorem as usual, $C= c$ and $S = s$. 
\end{proof} 

\medbreak 

Of course, it follows that if $z \in B_{\tfrac{1}{2} K} ( - \tfrac{1}{2} K)$ then $s(K + z) = c(- z)$ and $c(K + z) = s(- z)$. 

\medbreak

We take this opportunity to record the values of $s$ and $c$ at certain (nonzero) cardinal points encountered thus far. As noted at Theorem \ref{K}, $s( - \tfrac{1}{2} K) = -1$ and $c( - \tfrac{1}{2} K) = 2^{1/3}$ while $s(\tfrac{1}{2}K) = c(\tfrac{1}{2}K) = 2^{-1/3}$. Theorem \ref{gamma} completes these values to  
$$s(- \tfrac{1}{2} K) = \overline{\gamma} \, s(- \tfrac{1}{2} K \gamma) = \gamma \, s(- \tfrac{1}{2} K \overline{\gamma}) = - 1 \; \; {\rm and} \; \;  c(- \tfrac{1}{2} K) = c(- \tfrac{1}{2} K \gamma) = c(- \tfrac{1}{2} K \overline{\gamma}) = 2^{1/3}$$
and  
$$s(\tfrac{1}{2} K) = \overline{\gamma} \, s(\tfrac{1}{2} K \gamma) = \gamma \, s(\tfrac{1}{2} K \overline{\gamma}) = 2^{-1/3}\; \; {\rm and} \; \;  c(\tfrac{1}{2} K) = c(\tfrac{1}{2} K \gamma) = c(\tfrac{1}{2} K \overline{\gamma}) = 2^{-1/3}.$$
Finally, from the duplication formulae (or otherwise) we deduce   
$$s(K) =   \overline{\gamma} s(K \gamma) =  \gamma s(K \overline{\gamma}) = 1$$
and 
$$c(K) = c(K \gamma) = c(K \overline{\gamma}) = 0$$
along with the fact that $s$ and $c$ have poles at the points $- K$, $- K \gamma$ and $- K \overline{\gamma}$. 

\medbreak 

Incidentally, a consideration of the function inverse to the strictly increasing function $s|_{(-K, K)}$ reveals that the positive number $K$ naturally associated to $s$ and $c$ is given by  
$$K = \int_0^1 \frac{{\rm d} \sigma}{(1 - \sigma^3)^{2/3}} = 1.76663875...\, .$$

\bigbreak 

If we wish to continue the extension of $s$ and $c$ by reduplication, it would be appropriate to cut back a little and start afresh from the natural open disc $B_K (0)$: this is the largest disc about $0$ on which the functions $s$ and $c$ are holomorphic. 

\medbreak 

Rather than continue $s$ and $c$ by reduplication, we instead apply the Schwarz Symmetry Principle to complete their construction. We focus our attention on the regular hexagon ${\rm \bf H}$ having $\pm K, \; \pm K \gamma$ and $\pm K \overline{\gamma}$ as pairs of opposite vertices, the vertices in counterclockwise order being therefore 
$$K, \: - K \overline{\gamma}, \: K \gamma, \: - K, \: K \overline{\gamma}, - K \gamma.$$
Notice that the Schwarz Symmetry Principle recovers $s$ on this hexagon from $s$ on just the triangle $\bf \Delta$ with vertices $0, \: K $ and $\: - K \overline{\gamma}$ (for example). 

\medbreak 

The following result serves to initiate Schwarz reflexions of $s$ beyond the hexagon. 

\medbreak 

\begin{theorem} \label{hex} 
The values of $s$ along the edge of the hexagon {\rm \bf H} joining $K$ to $- K \overline{\gamma}$ are real.   
\end{theorem} 

\begin{proof} 
On account of the Identity Theorem, we need only check reality along the segment $(K, \tfrac{1}{2} K (1 - \overline{\gamma}))$  in which the interval $(K, - K \overline{\gamma})$ meets the disc $B_{K/2} (K / 2)$. Let $z$ be a point in this segment: by Theorem \ref{K -} and Theorem \ref{neg} we deduce that 
$$s(z) = c(K - z) = 1/c(z - K);$$
as $z - K \in \gamma \: \R$, the remarks after Theorem \ref{gamma} allow us to conclude that $s(z)$ is real. 
\end{proof} 

\medbreak 

Theorem \ref{real} implies that $s$ is real on the conjugate edge $(K, - K \gamma)$ of the hexagon; Theorem \ref{gamma} now implies that $s$ takes values in $\gamma \: \R$ along the edges that emanate from $K \gamma$ and in $\overline{\gamma} \: \R$ along the edges that emanate from $K \overline{\gamma}$. 

\medbreak 

Now, as $s$ is real-valued along the edge $(K, - K \overline{\gamma})$ of the hexagon {\rm \bf H}, the Schwarz Symmetry Principle analytically continues $s$ from the triangle $\bf \Delta$ across this edge to the triangle $\bf\Delta^*$ with vertices $K, \: - K \overline{\gamma}$ and $K (1 - \overline{\gamma})$: explicitly, if $z^* \in \bf\Delta^*$ is the reflexion across $(K, - K \overline{\gamma})$ of $z \in \Delta$ (so that $z^* = K (1 - \overline{\gamma}) + \overline{\gamma} \, \overline{z}$) then $s(z^*) = \overline{s(z)}$; thus, $s(0) = 0$ reflects to produce $s(K (1 - \overline{\gamma})) = 0$ while $s(- \tfrac{1}{2} K \overline{\gamma}) = - \overline{\gamma}$ yields $s(- K \overline{\gamma} + \tfrac{1}{2} K) = - \gamma$ and $s(\tfrac{1}{2} K) = 2^{-1/3}$ yields $s(K - \tfrac{1}{2} K \overline{\gamma}) = 2^{-1/3}$. The Schwarz Symmetry Principle likewise analytically continues $s$ across the remaining five edges of {\rm \bf H}, so that $s$ is extended holomorphically to the interior of a hexagram. For illustration, we offer just one further example of this: the function $s/ \overline{\gamma}$ is real-valued along the (open) edge of {\rm \bf H} that joins $- K$ to $K \overline{\gamma}$; the Schwarz Symmetry Principle therefore analytically continues $s$ across this edge so that $s(0) = 0$ reflects to $s(K(\overline{\gamma} - 1)) = 0$, $s(\tfrac{1}{2} K \overline{\gamma}) = \overline{\gamma} \: 2^{-1/3}$ to $s(K \overline{\gamma} - \tfrac{1}{2} K) = \overline{\gamma} \: 2^{-1/3}$ and $s(- \tfrac{1}{2} K) = -1$ to $s(- K + \tfrac{1}{2} K \overline{\gamma}) = - \gamma$. By the Identity Theorem, these continuations agree with $s$ as already defined at points of $B_{4 r} (0)$ outside {\rm \bf H}. 

\medbreak 

We leave to the reader the pleasure of verifying that repeated applications of the Schwarz Symmetry Principle extend $s$ holomorphically to the parallelogram {\rm \bf P} with vertices $0, \, - 3 K, \, 3 K \gamma$ and $3 K \overline{\gamma}$ except for poles at $- K$ (the first that we found), $2 K \gamma, \, 2 K \overline{\gamma}, \, K\gamma - 2K$ and $K \overline{\gamma} - 2K$; apart from poles, the values of $s$ lie in $\gamma \: \R$ along $[0, 3 K \gamma] \cup [-3 K, 3 K \overline{\gamma}]$ and in $\overline{\gamma} \: \R$ along $[0, 3 K \overline{\gamma}] \cup [-3 K, 3 K \gamma]$. We also leave as an exercise the verification that the pattern so formed repeats over the entire complex plane, revealing the fully extended $s$ as an elliptic function with {\rm \bf P} as a fundamental parallelogram. In particular, note that $s$ has as periods $3 K, \, 3 K \gamma$ and $3 K \overline{\gamma}$. Note also that the order of $s$ as an elliptic function is three: if {\rm \bf P} is translated slightly in the positive real direction, the shifted parallelogram includes only the simple poles at $- K, \, 2 K\gamma$ and $2 K \overline{\gamma}$; it may be checked that the residues of $s$ at these poles are $-1, \, - \overline{\gamma}$ and $- \gamma$ respectively (with sum zero, as it should be). Alternatively, the shifted parallelogram includes simple zeros at $0, \, - K + K \gamma$ and $- K + K \overline{\gamma}$.  

\medbreak 

The holomorphic function $c : B_K (0) \to \C$ also extends fully to an elliptic function,  for which we continue the notation $c$. Perhaps the swiftest justification of this claim takes its cue from Theorem \ref{K -} in conjunction with the Identity Theorem and simply defines $c(z) : = s(K - z)$ whenever $z \in \C$ does not lie in the pole-set of $s$ (which the symmetry $z \mapsto K - z$ leaves invariant). Instead, $c$ also may be extended `kaleidoscopically' via the Schwarz Symmetry Principle, starting from the following counterpart to Theorem \ref{hex}. 

\medbreak 

\begin{theorem} \label{hexc}
The values of $c$ along the edge of {\rm \bf H} joining $K$ to $- K \overline{\gamma}$  lie in $\gamma \: \R$. 
\end{theorem} 

\begin{proof} 
Let $z$ lie on the indicated edge: thus, $z = K (1 - \overline{\gamma}) + \overline{\gamma} \, \overline{z}$ and so $K - \overline{z} = \gamma (K - z)$. Now 
$$\overline{c(z)} = \overline{s(K - z)} = s(K - \overline{z}) = s(\gamma( K - z)) = \gamma s(K - z) = \gamma c(z)$$
by Theorem \ref{real} and Theorem \ref{gamma} along with Theorem \ref{K -} (taking $z$ inside $B_{K/2} (K/2)$ as we did for Theorem \ref{hex}). The condition $\overline{c(z)} = \gamma c(z)$ places $c(z)$ on the line $\gamma \: \R$. 
\end{proof} 

\medbreak 

Theorem \ref{real} and Theorem \ref{gamma} now imply that the values of $c$ on the edges of {\rm \bf H} lie alternately in $\gamma \: \R$ and in $\overline{\gamma} \: \R$.

\medbreak 

Naturally, the Identity Theorem ensures that the elliptic functions $s$ and $c$ continue to satisfy {\bf IVP} and the relation $c^3 + s^3 = 1$; also the various symmetries that were established for the ancestral holomorphic $s$ and $c$ in Theorem \ref{real}, Theorem \ref{neg} and Theorem \ref{gamma}; also the `trigonometric' identity of Theorem \ref{K -}. Further identities follow from these in combination: for example, 
$$c(2 K + z) = s(2 K - z) = 1/s(z)\; \; {\rm and} \; \; s(2 K + z) = c(2 K - z) = - c(z)/s(z);$$
in addition, we may recover the fact that $3 K, \, 3 K \gamma$ and $3 K \overline{\gamma}$ are periods of $c$ and $s$.   

\medbreak 

\begin{theorem} \label{T}
The function $s$ maps the edges of the triangle with vertices $K, \, K \gamma$ and $K \overline{\gamma}$ to the unit circle. 
\end{theorem} 

\begin{proof} 

Let $z$ lie on the edge that joins $K \gamma$ to $K \overline{\gamma}$: thus $\overline{z} = - K - z$ and so  
$$\overline{s(z)} = s(\overline{z}) = s(- K - z) = - s(K + z) / c(K + z)$$
by Theorem \ref{real} and Theorem \ref{neg}; here 
$$s(K + z) = c(-z) = 1/c(z) \; \; {\rm and} \; \; c(K + z) = s(-z) = - s(z)/c(z)$$
by Theorem \ref{K -} and a further application of Theorem \ref{neg}, whence $\overline{s(z)} = 1/s(z)$ and therefore $|s(z)| = 1$. This proves that $s$ maps the edge $[K \gamma, K \overline{\gamma}]$ to the unit circle; the symmetry of $s$ expressed in Theorem \ref{gamma} concludes the proof. 
\end{proof} 

\medbreak 

We remark that a proof is possible earlier than this, by rotating about $0$ the middle third (say) of the edge $[K, K \gamma]$ so as to place it along the line ${\rm Re} = \tfrac{1}{2} K$ and inside the disc $B_{K/2}(K/2)$; but the present proof is cleaner.  

\medbreak 

Incidentally, the function $c$ clearly maps the entire imaginary axis to the unit circle: if $\overline{z} = -z$ then $\overline{c(z)} = c(\overline{z}) = c(-z) = 1/c(z)$ by Theorem \ref{real} and Theorem \ref{neg}. 

\medbreak 

The elliptic functions $s$ and $c$ have many other properties. Perhaps the most important of these are their addition formulae, which take many equivalent forms. One form is the pair 

$$c(a + z) = \frac{c(a) c(z)^2 - s(a)^2 s(z)}{s(a) c(a) s(z)^2 + c(z)}$$ 

$$s(a + z) = \frac{s(a) + c(a)^2 s(z) c(z)}{s(a) c(a) s(z)^2 + c(z)}$$

\medbreak 
\noindent 
which may be verified (not discovered!) by fixing $a$ but varying $z$ and checking that the functions of $z$ on the two sides satisfy the same differential equations $C\,' = - S^2, \; S\,' = C^2$  and the same initial conditions $C(0) = c(a), \: S(0) = s(a)$. 

\medbreak 

Specialization of these addition formulae naturally recovers the duplication formulae that played a significant part in our construction of $c$ and $s$. Further application then leads to the triplication formulae 

$$s(3 z) = s(z) c(z) \frac{2 + c(z)^6 - s(z)^3 c(z)^3 + s(z)^6}{c(z)^3 - s(z)^6 + 3 s(z)^3 c(z)^3 + s(z)^3 c(z)^6}\, ,$$

$$c(3 z) = \frac{c(z)^6 - s(z)^3 - 3 s(z)^3 c(z)^3 - s(z)^6 c(z)^3}{c(z)^3 - s(z)^6 + 3 s(z)^3 c(z)^3 + s(z)^3 c(z)^6}\, .$$

\medbreak 
\noindent 
Here, the formula for $c(3 z)$ corrects apparent sign errors in formula (44) of [2]. It is possible to extend $c$ and $s$ beyond $B_r (0)$ by triplication rather than by duplication (and the poles are secured by just one triplication); the details are interesting but challenging. 

\medbreak 

Finally, it is appropriate to mention the Weierstrass function $\wp$ that is associated to the period lattice of $s$ and $c$. In [1] the expression for the relevant $\wp$ is attributed to Dumont: explicitly, $\wp$ is given by 
$$3 \: \wp = \frac{s}{1 - c}$$
\medbreak 
\noindent 
with $3 K, \, 3 K \gamma$ and $3 K \overline{\gamma}$ as periods. It may be checked by differentiation that 
$$3 \: \wp\,' = \frac{c + 1}{c - 1}$$
and therefore that 
$$27 \: ((\wp\,')^2 - 4 \: \wp^3) = - 1$$
or  
$$( \, \wp\,')^2 = 4 \, \wp^3 -1/27.$$ 
\medbreak 
\noindent 
In the (shifted) period parallelogram {\rm \bf P}, $\wp$ has just one pole: a double pole at the origin, where the simple zero of $s$ is outmatched by the triple zero of $1 - c$; the simple zeros of $s$ at $- K + K \gamma$ and $- K + K \overline{\gamma}$ still survive as simple zeros of $\wp$. Thus, $\wp$ is indeed the Weierstrass function with {\rm \bf P} as period parallelogram. From the first-order equation that it satisfies, $\wp$ has invariants $g_2 = 0$ and $g_3 = 1/27$. The formulae for $\wp$ and $\wp\,'$ displayed above may be solved for $c$ and $s$: thus 
$$c = \frac{3\: \wp\,' + 1}{3\: \wp\,' - 1} \; \; {\rm and} \; \; s = \frac{6\: \wp}{1 - 3\: \wp\,'}\: .$$
\medbreak 
\noindent 
In the opposite direction, when the Weierstrass function $\wp$ with the indicated invariants is given, these formulae may be taken to define the Dixonian elliptic functions $c$ and $s$. 

\medbreak 

\bigbreak 

\begin{center} 
{\small R}{\footnotesize EFERENCES}
\end{center} 
\medbreak 

[1] E. Conrad and P. Flajolet, {\it The Fermat cubic, elliptic functions, continued fractions, and a combinatorial excursion}, S\'eminaire Lotharingien de Combinatoire 54 (2006) Article B54g.

\medbreak 

[2] A.C. Dixon, {\it On the doubly periodic functions arising out of the curve $x^3 + y^3 - 3 \alpha x y = 1$}, The Quarterly Journal of Pure and Applied Mathematics, 24 (1890) 167–233. 

\medbreak 

[3] E. Hille, {\it Ordinary Differential Equations in the Complex Domain}, Wiley-Interscience (1976); Dover Publications (1997).

\medbreak 

[4] J.C. Langer and D.A. Singer, {\it The Trefoil}, Milan Journal of Mathematics, 82 (2014) 161-182. 

\medbreak 

[5] E.H. Neville, {\it Jacobian Elliptic Functions}, Oxford University Press (1944). 

\medbreak

\end{document}